\documentclass[12pt]{amsart}
\usepackage{amssymb,latexsym,amsfonts}
\usepackage{amsmath}
\usepackage{amscd}
\usepackage{pb-diagram,pb-xy}

\usepackage{graphicx}
\usepackage{hyperref}
\usepackage{color}
\usepackage[matrix,arrow]{xy}
\newcommand{\incircle}{\raisebox{1pt}{\textcircled{\raisebox{-.9pt} {r}}}}
\newcommand{\iincircle}{\raisebox{1pt}{\textcircled{\raisebox{-.9pt} {i}}}}
\newcommand{\ccc}{{\rm c}}

\DeclareMathAlphabet{\mathpzc}{OT1}{pzc}{m}{it}

\newtheorem{theorem}{Theorem}[section]
\newtheorem{corollary}[theorem]{Corollary}
\newtheorem{lemma}[theorem]{Lemma}
\newtheorem{proposition}[theorem]{Proposition}
\newtheorem{thm}{Theorem}[section]
\theoremstyle{definition}
\newtheorem{definition}[thm]{Definition}

\newtheorem*{remark}{Remark}

\begin{document}

 \author{Mohammed Larbi  Labbi}
 \title[Algebraic Identities]{ On some algebraic identities and the exterior product of double forms}
   \date{}
\thanks{This research is funded by the Deanship of Scientific Research at the University of Bahrain ref. 3/2010.}
\begin{abstract} We use the exterior product of double forms to reformulate celebrated classical results of linear algebra about matrices and bilinear forms namely Cayley-Hamilton theorem, Laplace expansion of the determinant, Newton identities and Jacobi's formula for the determinant. This new formalism is then used to naturally  generalize the previous results higher multilinear forms namely  to  double forms. \\
 In particular, we show that the Cayley-Hamilton theorem once applied to the second fundamental form of a hypersurface is equivalent to a linearized version of the Gauss-Bonnet theorem, and once its  generalization is applied to the Riemann curvature tensor (seen as a  $(2,2)$ double form) is an infinitisimal version of the  general Gauss-Bonnet-Chern  theorem. In addition to that,  the general Cayley-Hamilton theorems generate  several universal curvature identities. The generalization of the classical Laplace expansion of the determinant  to double forms is shown to lead to new  general Avez type formulas for all Gauss-Bonnet  curvatures. 
\end{abstract}
   \maketitle

\tableofcontents
\section{Preliminaries: The Algebra of Double Forms}\label{prem.}
\subsection{Definitions and basic properties}
For the convenience of the reader, we start by recalling some  basic facts about  the exterior product of double forms. For further study  and for the proofs,  the reader is  invited to consult \cite{Kulkarni, Labbidoubleforms}.\\
Let $(V,g)$ be an Euclidean real vector space  of dimension n. In the
 following
we shall identify whenever convenient (via their Euclidean structures),
the vector spaces
 with their duals. Let
  $\Lambda V^{*}=\bigoplus_{p\geq 0}\Lambda^{p}V^{*}$  (resp.
   $\Lambda V=\bigoplus_{p\geq 0}\Lambda^{p}V$) denotes the exterior algebra
 of $V^* $ (resp.   $V$). Considering  tensor products,
  we define the space of double forms as
 $${\mathcal D}= \Lambda V^{*}\otimes \Lambda V^{*}=\bigoplus_{p,q\geq 0}
  {\mathcal D}^{p,q},$$
 where $  {\mathcal D}^{p,q}= \Lambda^{p}V^{*} \otimes  \Lambda^{q}V^{*}.$
The space ${\mathcal D}$ is naturally  a bi-graded associative  algebra,
    where 
    for $\omega_1=\theta_1\otimes \theta_2\in { \mathcal D}^{p,q}$ and
    $\omega_2=\theta_3\otimes \theta_4\in  {\mathcal D}^{r,s}$, the multiplication is given  by 
    \begin{equation}
    \label{def:prod}
     \omega_1\omega_2= (\theta_1\otimes \theta_2 )(\theta_3\otimes
     \theta_4)=
    (\theta_1\wedge \theta_3 )\otimes(\theta_2\wedge \theta_4)\in
    {\mathcal D}^{p+r,q+s}.\end{equation}
Where $\wedge$ denotes the standard exterior product on the associative exterior algebra $ \Lambda V^{*}$. It results directly from the definition that  the exterior product is (anti)-commutative in the following sense\\
 \begin{equation}
\omega_1\omega_2=(-1)^{pr+qs}\omega_2.\omega_1.
\end{equation}
A \emph{ $(p,q)$ double form}  is by definition an  element of the tensor product
    $  {\mathcal D}^{p,q}= \Lambda^{p}V^* \otimes \Lambda^{q}V^*$. 
     It can be identified canonically
     with a bilinear form $\Lambda^pV\times\Lambda^qV\rightarrow {\bf R}$, which in turn can be seen as 
     a multilinear form which is skew symmetric in the first $p$-arguments and also
     in the last $q$-arguments.\\
The above multiplication in ${\mathcal D}$  is called the \emph{exterior product of double forms.} 
 It turns out that  the exterior product of two ordinary bilinear forms on $V$
coincides with the celebrated  Kulkarni-Nomizu product of bilinear forms. Furthermore,  the  $k$-th exterior power of a
  bilinear form  $h$ on $V$ is a  double form of degree $(k,k)$ that  is given by the determinant as follows
\begin{equation}
h^k(x_1 \wedge...\wedge x_k,y_1\wedge...\wedge y_k)=k!\det[h(x_i,y_j)].
\end{equation}
In particular, for $h=g$ and for each $1\leq k\leq n$,  $\frac{g^k}{k!}$ coincides with  the canonical inner product on $\Lambda^{k}V$. The former canonical inner
 product extends to an inner product on the exterior algebra of $V$, which in turn can be extended in a natural way to an inner product on the algebra of double forms over $V$. 
The so obtained  inner product of double forms  shall be denoted by $\langle, \rangle$. Explicitely, for $\omega_1=\theta_1\otimes \theta_2\in { \mathcal D}^{p,q}$ and
    $\omega_2=\theta_3\otimes \theta_4\in  {\mathcal D}^{r,s}$, we have 
    \begin{equation}\label{def:innerpdct}
\langle \omega_1, \omega_2 \rangle =\langle \theta_1, \theta_3 \rangle \langle \theta_2, \theta_4 \rangle =\frac{g^k}{k!} ( \theta_1^\sharp, \theta_3^\sharp)    \frac{g^k}{k!} ( \theta_2^\sharp, \theta_4^\sharp).
    \end{equation}
 Where   $ \theta_i^\sharp$  denotes the  $p$-vector dual to the $p$-form  $\theta_i$.\\
Recall that the (Ricci) contraction map, denoted by $\ccc$,  maps  ${\mathcal D}^{p,q}$ into  ${\mathcal D}^{p-1,q-1}$. For a double form  $\omega \in  {\mathcal D}^{p,q}$  with $p\geq 1$ and $q\geq 1$, we  have
$$\ccc \omega(x_1\wedge...\wedge x_{p-1},y_1\wedge...\wedge y_{q-1})=
\sum_{j=1}^{n}\omega(e_j\wedge x_1\wedge... x_{p-1},
e_j\wedge y_1\wedge...\wedge y_{q-1})$$
where $\left\{e_1,...,e_n\right\}$ is an arbitrary  orthonormal basis of $V$ and $\omega$ is considered as a bilinear form as explained above.\\
It turns out that the contraction map $\ccc$ on ${\mathcal D}$  is  the adjoint of  the multiplication map by the metric  $g$ of $V$,  precisely we have for $\omega_1, \omega_2 \in  {\mathcal D}$ the following \cite{Labbidoubleforms}
\begin{equation}\label{adj:gc}
<g\omega_1,\omega_2>=<\omega_1,\ccc \omega_2>.
\end{equation}
Suppose now that  we have fixed an orientation on the vector space  $V$.
The classical
Hodge star operator  $*:\Lambda^{p}V\rightarrow \Lambda^{n-p}V$ can be extended naturally to
 operate 
on double forms as follows. For a $(p,q)$-double form $\omega$ (seen as a bilinear form), $*\omega$  is  the  $(n-p,n-q)$-double form  given by
\begin{equation}
 *\omega(.,.)=(-1)^{(p+q)(n-p-q)}\omega(*.,*.).
\end{equation}
Note that $*\omega$ does not depend on the chosen orientation as the usual Hodge star
 operator is applied twice.
 The so-obtained  operator is still called the Hodge star operator operating on double forms or the \emph{double Hodge star operator}. This new operator  provides another simple  relation between the contraction map $c$ of double forms 
and the multiplication map by the metric as follows \cite{Labbidoubleforms}
\begin{equation}
\label{A}
g\omega=*\ccc *\omega\, \, {\rm and}\, \,  \ccc \omega =*g*\omega.
\end{equation}
Furthermore, the double Hodge star operator generates the inner product of double forms  as follows.  For any two double forms $\omega,\theta\in {\mathcal D}^{p,q}$ we have
\begin{equation}\label{B}
<\omega,\theta>=*\bigl(\omega(*\theta)\bigr)=(-1)^{(p+q)(n-p-q)}*\bigl((*\omega)\theta\bigr).
\end{equation}
Finally, let us recall that a double form  $\omega\in {\mathcal D}^{p,q}$ is said to  satisfy the first Bianchi if
$$\sum_{j=1}^{p+1}(-1)^j\omega(x_1\wedge...\wedge \hat{x}_j\wedge ... x_{p+1},x_j\wedge y_1\wedge...\wedge y_{q-1})=0.$$
For all vectors $x_1,...,x_{p+1},y_1,...,y_p$ in $V$  and where $\hat{}$ denotes omission.\\
It turns out that the exterior product of two double forms satisfying the first Bianchi iidentity is a double form that satisfies the first Bianchi identity as well \cite{Kulkarni}.
For  a double form $\omega$ of degree $(p,p)$ that satisfies the first Bianchi identity, we have the following useful relations \cite{Labbidoubleforms}
\begin{equation}\label{star:form}
{1\over (k-p)!}*(g^{k-p}\omega)=\sum_{r={\max\{0,p-n+k\}}}^p
{(-1)^{r+p}\over r!}{g^{n-k-p+r}\over (n-k-p+r)!}\ccc^r\omega
\end{equation}
Where $1\leq p\leq k\leq n$. In particular for $k=n$ and $k=n-1$ respectively, we have
\begin{equation}\label{nn-1:star}
*({g^{n-p}\omega\over (n-p)!})={1\over p!}\ccc^p\omega\quad {\rm and}\,
*({g^{n-p-1}\omega\over (n-p-1)!})={\ccc^p\omega\over p!} g-
{\ccc^{p-1}\over (p-1)!}\omega.
\end{equation}
\subsection{ The algebra of double forms vs. Mixed exterior algebra}
The algebra of double forms was considered and studied since the sixties of the last century exclusively by geometers. However, as was pointed out recently  by  Jammes in his habilitation thesis \cite{Jammes}, it seems that  geometers  ignore   that many algebraic aspects of this algebra was indirectly and  independently  studied   by  Greub in the sixties and later by Greub and  Vanstone,  under the name of Mixed exterior algebra, see for instance \cite{Greub-book, Vanstone}. The first edition of Greub's book \cite{Greub-book}  appeared in 1967.  The author just came to know about Greub and Vanstone's contributions and  only after he finishes the first version of this paper.  Let us recall here the basic definition of this "dual"  algebra, a  report on Greub and Vanstone  contributions will appear in a forthcoming paper.\\
Let $V$ and $V^*$  be two dual vector spaces, denote by  $\Lambda^p  V$ and  $\Lambda^q V^{*}$  their exterior powers respectively. Consider the tensor product  $$\Lambda^p_q ( V, V^*)= \Lambda^p  V\otimes \Lambda^q V^{*},$$
An element in $\Lambda^p_q ( V, V^*)$ is called  a $(p,q)$-vector (the analogous of a $(p,q)$-double form). Next define the mixed exterior algebra as the tensor product
 $$\Lambda ( V, V^*)=\Lambda V\otimes \Lambda V^{*}=\bigoplus_{p,q\geq 0} \Lambda^p_q ( V, V^*).$$
Where the multiplication is denoted by a wedge and  given by
$$(u\otimes u^*)\wedge (v\otimes v^*)=(u\wedge v)\otimes (u^*\wedge v^*).$$
The space $\Lambda ( V, V^*)$ is  isomorphic to the space of linear endomorphisms of $\Lambda ( V)$. Greub and Vanstone introduced then a second product, which they call the composition product,   in the algebra $\Lambda ( V, V^*)$ by just  pulling back the standard composition operation on endomorphisms. They  proved useful identities between the two products on $\Lambda ( V, V^*)$. These identities were then used  to obtain several  matrix-free proofs of classical theorems about linear transformations similar to the ones that we  prove here in this paper  in section 2.  However, let us emphasize here that the results of the remaining sections of our paper are original and the corresponding cases were not discussed by Greub and Vanstone.
\subsection{The composition product of double forms}
Following Greub \cite{Greub-book}, we define a second multiplication in the space of double forms ${\mathcal D}$ which will be denoted by $\circ$ and will be called the \emph{composition product} or \emph{Greub's product} of double forms. Given 
 $\omega_1=\theta_1\otimes \theta_2\in { \mathcal D}^{p,q}$ and
    $\omega_2=\theta_3\otimes \theta_4\in  {\mathcal D}^{r,s}$, set
    \begin{equation}
\omega_1\circ\omega_2=(\theta_1\otimes \theta_2)\circ (\theta_3\otimes \theta_4)=\langle \theta_1,\theta_4\rangle \theta_3\otimes \theta_2\in  {\mathcal D}^{r,q}.
\end{equation}
It is clear that $\omega_1\circ\omega_2=0$ unless $p=s$.\\
This product can be interpreted in the following way: Denote by $\bar\omega_1:\Lambda^p  V\rightarrow \Lambda^q  V$   the linear map corresponding to the bilinear map  $\omega_1\in  {\mathcal D}^{p,q}$, and by  $\bar\omega_2:\Lambda^r  V\rightarrow \Lambda^p  V$   the linear map corresponding to the double form  $\omega_2\in  {\mathcal D}^{r,p}$. Then it turns out that the composition map $\bar\omega_1\circ\bar\omega_2:\Lambda^r  V\rightarrow \Lambda^q  V$ is nothing but the linear map corresponding to the double form $\omega_1\circ\omega_2\in  {\mathcal D}^{r,q}$. The space of double forms endowed with the composition product  $\circ$ is then an associative algebra.\\
We are going now to write an explicit useful formula for this new product.  Let  $u_1\in \Lambda^r$  be an $r$-vector and $u_2\in \Lambda^q$ a $q$-vector in $V$  then

\begin{equation*}
\begin{split}
\omega_1\circ\omega_2(u_1,u_2)=& \langle \bar{\omega}_1\circ \bar{\omega}_2 (u_1), u_2 \rangle\\
=& \langle \bar\omega_1(\bar\omega_2(u_1)),u_2\rangle\\
=& \omega_1(\bar\omega_2(u_1),u_2)\\
=&\sum_{\scriptstyle i_1<i_2<...<i_p}\omega_2(u_1,e_{i_1}\wedge ...\wedge e_{i_p})\omega_1(e_{i_1}\wedge ...\wedge e_{i_p},u_2).
\end{split}
\end{equation*}
Where $\{e_1,...,e_n\}$ is an arbitrary orthonormal basis of $(V,g)$. For a double form $\omega \in  {\mathcal D}^{p,q}$, we denote by $\omega^t\in  {\mathcal D}^{q,p}$ the transpose of $\omega$, that is 
\begin{equation}
\omega^t(u_1,u_2)=\omega(u_2,u_1).
\end{equation} 
In particular, $\omega$ is a symmetric double form if and only if $\omega^t=\omega$. The previous calculation shows that
\begin{equation}\label{compo-product}
\omega_1\circ\omega_2(u_1,u_2)=\sum_{\scriptstyle i_1<i_2<...<i_p}\omega_2^t(e_{i_1}\wedge ...\wedge e_{i_p},u_1)\omega_1(e_{i_1}\wedge ...\wedge e_{i_p},u_2).
\end{equation}
Consequently, we obtain another useful formula for the inner product of double forms as follows
\begin{proposition}
The inner product of two double forms $\omega_1,\omega_2\in  {\mathcal D}^{p,q}$ is the full contraction of the product $\omega_1^t\circ\omega_2$ or $\omega_2^t\circ\omega_1$ , precisely we have
\begin{equation}
\langle \omega_1,\omega_2\rangle=\frac{1}{p!}\ccc^p(\omega_2^t\circ\omega_1)=\frac{1}{p!}\ccc^p(\omega_1^t\circ\omega_2.)
\end{equation}
\end{proposition}
\begin{proof}
It is straightforward as follows:
\begin{equation*}
\begin{split}
\ccc^p(\omega_2^t\circ\omega_1) &=\sum_{\scriptstyle i_1,i_2,...,i_p}\omega_2^t\circ \omega_1(e_{i_1}\wedge ...\wedge e_{i_p};e_{i_1}\wedge ...\wedge e_{i_p})\\
&=p!\sum_{\scriptstyle i_1<i_2<...<i_p}\omega_2^t\circ \omega_1(e_{i_1}\wedge ...\wedge e_{i_p};e_{i_1}\wedge ...\wedge e_{i_p})\\
&=p!\sum_{\scriptstyle i_1<i_2<...<i_p\atop \scriptstyle  j_1<j_2<...<j_p} \omega_1^t(e_{j_1}\wedge ...\wedge e_{j_p};e_{i_1}\wedge ...\wedge e_{i_p})\omega_2^t(e_{j_1}\wedge ...\wedge e_{j_p};e_{i_1}\wedge ...\wedge e_{i_p})\\
&=p!\sum_{\scriptstyle i_1<i_2<...<i_p\atop \scriptstyle  j_1<j_2<...<j_p} \omega_1(e_{i_1}\wedge ...\wedge e_{i_p};e_{j_1}\wedge ...\wedge e_{j_p})\omega_2(e_{i_1}\wedge ...\wedge e_{i_p};e_{j_1}\wedge ...\wedge e_{j_p})\\
&=p!\langle \omega_1,\omega_2\rangle.
\end{split}
\end{equation*}
\end{proof}
\begin{remark}
The inner product used by Greub in his mixed exterior algebra is the pairing product which is always  non-degenerate but not positive definite in general, hence  it is different from the above inner product. The two inner products  coincide only for symmetric double forms. For  two general double forms $\omega_1,\omega_2\in  {\mathcal D}^{p,q}$,  their  pairing product  is the full contraction of the  product double form  $\omega_1\circ\omega_2$.
\end{remark}
Following Greub \cite{Greub-book}, in what follows  we shall denote the $k$-th power of a double form $\omega$  in the composition algebra  by $\omega^{ \incircle}$, that is 
$$\omega^{ \incircle}=\underbrace{\omega\circ ... \circ \omega}_\text{$r$-times}.$$

\section{Euclidean invariants for  bilinear forms}\label{first-section}
\subsection{characteristic coefficients of bilinear forms}
Let $A$ be a square matrix with real  entries of size $n$. Recall that the characteristic polynomial $\chi_A(\lambda)$ of $A$ is given by
\begin{equation}\label{skdef}
\chi_{A} (\lambda)=\det(A-\lambda I)=(-1)^n\lambda^n+(-1)^{n-1}s_1(A)\lambda^{n-1}+...+s_n(A)=\sum_{i=0}^n (-1)^{n-i} s_i(A)\lambda^{n-i}.
\end{equation}
Where $s_1(A)$ is the trace of $A$, $s_n(A)$ is the determinant of $A$ and the other characteristic  coefficients $s_k(A)$  are intermediate invariants of the matrix $A$ that interpolate between the trace and the detreminant, we shall call them here for simplicity the $s_k$ invariants of the matrix $A$. \\
Since similar matrices have the same  characteristic polynomial, therefore they have   as well the same $s_k$ invariants. In particular, one can define these invariants in an invariant way for endomorphisms.\\
Let $(V,g)$ be an Euclidean vector space of finite dimension $n$ and  $h$ be a  bilinear form on $V$. We denote by $\bar{h}$  the linear operator on $V$  that coresponds to $h$  via the inner product $g$.\\
We define  the $s_k$ invariants of the bilinear form  $h$ to be those of the linear operator $\bar{h}$. In particular, the determinant of the bilinear form $h$ is by definition the determinant of the linear operator $\bar{h}$.  Note that in contrast with $s_k(\bar{h})$, the invariants $s_k(h)$ depend on  the inner product $g$. In order to make this dependence explicit we shall use the exterior product of double forms .\\
Recall that  for $1\leq k\leq n$, the exterior $k$-th power $h^k=h...h$  of the  bilinear form $h$  (seen here as a $(1,1)$-double form)  is a $(k,k)$ double form determined by the determinant as follows:
\begin{equation}\label{hkdet}
h^k(x_1,..., x_k; y_1,..., y_k)=k!\det[h(x_i,y_j)].
\end{equation}
In particular, if $\{e_1,...,e_n\}$ is an orthonormal basis of $(V,g)$ then then  we have
\[h^n(e_1,...,e_n;e_1,...,e_n)=n!\det h.\]
Alternatively, this can be written as
\begin{equation}\label{char-sn}
h^n(*1,*1)=n!\det h, \,\, {\rm{or}}\,\, 
\det h=*\frac{h^n}{n!}.
\end{equation}
Where $*$ denotes  the  (double) Hodge star operator  as in the previous section.
 Consequently, using the binomial formula, the characteristic polynomial of $h$ takes the form
\begin{equation}
\begin{split}
\chi_{h}(\lambda)=&\det(h-\lambda g)=*\frac{(h-\lambda g)^n}{n!}\\
=&*\frac{1}{n!}\sum_{i=0}^n\binom{n}{i}h^i(-1)^{n-i}\lambda^{n-i}g^{n-i}\\
=& \frac{1}{n!}\sum_{i=0}^n n!(-1)^{n-i} \left(*\frac{g^{n-i} h^i}{(n-i)!i!}\right)\lambda^{n-i}\\
=&\sum_{i=0}^n (-1)^{n-i} s_i(h)\lambda^{n-i}.
\end{split}
\end{equation}
We have  therefore proved the  following simple formula for all the $s_k$ invariants of $h$:
\begin{proposition}\label{sk-prop}
For each $1\leq k\leq n$, the  $s_k$ invariant of $h$ is given by
\begin{equation}\label{skformula}
s_k(h)= \frac{1}{k!(n-k)!}*(g^{n-k}h^k).
\end{equation}
Where $*$ denotes
 the  (double) Hodge star operator operating on double forms, and the products $g^{n-k}, h^k,g^{n-k}h^k$  are exterior products of double forms, where $g$ and $h$ are considered as $(1,1)$-double forms.
\end{proposition}
In particular, the trace and determinant of $h$  are given by
\begin{equation*}
s_1(h)=*\left\{\frac{g^{n-1}h}{(n-1)!}\right\} \, \, 
{\rm and}\,\,  s_n(h)=*\frac{h^n}{n!}.
\end{equation*}
Let us note  here that for any  orthonormal basis $(e_i)$ of $(V,g)$,  $s_k(h)$ coincides by definition  with the $s_k$ invariant of the  matrix $(h(e_i,e_j))$. In particular, $s_n(h)$ is the determinant of the matrix  $(h(e_i,e_j))$. More generally, we have the following  lemma:
\begin{lemma}
Let  $(e_i), \,  i=1, 2, ... n$, be an  orthonormal basis of $(V,g)$, $h$ a bilinear form on $V$  and  $k+r\leq n$. Then for any subset 
$\{i_1,i_2,...,i_{k+r}\}$   with $k+r$ elements of $\{1,2,...,n\}$ we have
\begin{equation}\label{gkhrsr}
g^kh^r\left(e_{i_1},e_{i_2},...,e_{ i_{k+r}},e_{i_1},e_{i_2},...,e_{ i_{k+r}}\right)=k!r!s_r\left(h(e_{i_a},e_{i_b})\right).
\end{equation}
Where $ s_r\left(h(e_{i_a},e_{i_b})\right)$  is the $s_r$ invariant of the $(k+r)\times (k+r)$ matrix $\left(h(e_{i_a},e_{i_b})\right)$, for $1\leq a,b\leq k+r$,  $g$ is the inner product on $V$. The product $g^kh^r$ is the exterior product of double forms as above.
\end{lemma}

The proof of the previous lemma is a direct consequence of proposition \ref{sk-prop}. The lemma shows that the  "sectionnal curvatures" of the tensors $g^kh^r$ are precisely the $s_r$ invariants of $h$ once restricted to lower subspaces of $V$. \\
In what follows in this section  we shall  use this formalism of the $s_k$ invariants to reformulate and then generalize  celebrated classical identities for matrices namely Laplace expansion of the determinant, Cayley-Hamilton theorem, Jacobi's formula for the determinant and Newton identities.
\subsection{Cofactor transformation of a bilinear form vs. cofactor matrix}\label{adjugate}
Recall that the  cofactor matrix of a square matrix of size $n$  is a new matrix formed by all the cofactors of the original matrix. Where  the $(ij)$-cofactor of a matrix is $(-1)^{i+j}$ times the  determinant of the $(n-1\times n-1)$ sub-matrix that is obtained by eliminating the $i$-th row and $j$-th column of the original matrix.\\
 We are going to do  here the same transformation but on a bilinear form instead of a matrix. Precisely, let $(e_i)$ be an orthonormal basis of $(V,g)$, for each pair of  indexes $(i,j)$, we define,  the $(ij)$-cofactor   of the bilinear form $h$, denoted $t_{n-1}(h)(e_i,e_j)$,
to be  $(-1)^{i+j}$ multiplied by     the determinant of the  $(n-1\times n-1)$ sub-matrix  that is obtained after removing the $i$-th row and $j$-th column from the matrix  $(h(e_i,e_j))$. Then we can use bilinearity to extend $t_{n-1}(h)$ into a  bilinear form defined on $V$. It  is then natural to call the so obtained bilinear form  the cofactor transformation of $h$. \\
The next proposition shows in particular that the bilinear form $t_{n-1}(h)$ is well defined (that is it does not depend on choice of the orthonormal basis)
\begin{proposition} If 
$*$ denotes the Hodge star operator operating on double forms then
\begin{equation}\label{tkformula} 
t_{n-1}(h)=\frac{1}{(n-1)!}*\bigl(h^{n-1}\bigr).
\end{equation}

\end{proposition}
\begin{proof}
From the  definition of the Hodge star operator once it is acting on double forms  we have
\begin{equation}
\frac{1}{(n-1)!}*\bigl(h^{n-1}\bigr)(e_i,e_j)=\frac{1}{(n-1)!}\bigl(h^{n-1}\bigr)(*e_i,*e_j).
\end{equation}
The last expression is by formula (\ref{hkdet}) exactly equal to $(-1)^{i+j}$   multiplied by  the determinant of the $(n-1)\times (n-1)$ sub-matrix  that is obtained from $(h(e_i,e_j))$  by removing the $i$-th row and $j$-th column. This completes the proof of the proposition.
\end{proof}
We define now higher cofactor transformations of $h$ as follows
\begin{definition}
For $0\leq k\leq n-1$ we define the $k$-th cofactor of $h$  (called also the $k$-th Newton transformation of $h$) to be the bilinear form given by
\begin{equation}\label{tkformula}
t_k(h)=\frac{1}{k!\bigl(n-1-k\bigr)!}*\bigl(g^{n-1-k}h^k\bigr).
\end{equation}
\end{definition}
Note that $t_0(h)=g$ is the metric, $t_{n-1}(h)$ coincides with the above defined cofactor transformation of $h$. The terminology  ``higher  cofactor"  is motivated by the following fact 
\begin{equation*}
t_k(h)(e_i,e_j)=\frac{1}{k!\bigl(n-1-k\bigr)!}\bigl(g^{n-1-k}h^k\bigr)(*e_i,*e_j).
\end{equation*}
The value $\frac{1}{k!\bigl(n-1-k\bigr)!}\bigl(g^{n-1-k}h^k\bigr)(*e_i,*e_j)$ can be seen, like in the case $k=n-1$ above, as $(-1)^{i+j}$   multiplied by  a kind of higher determinant  of the $(n-1)\times (n-1)$ sub-matrix  that is obtained from $(h(e_i,e_j))$  by removing the $i$-th row and $j$-th column.
\begin{remark}
In view of formula  \ref{gkhrsr} it is tempting to think that the  ``higher determinants"   $\frac{1}{k!\bigl(n-1-k\bigr)!}\bigl(g^{n-1-k}h^k\bigr)(*e_i,*e_j)$  coincide with the $s_k$ invariant of the corresponding matrix. However it turns out that this not the case in general even when $h$ is symmetric.
\end{remark}
Recall that the $s_k$ invariants of a bilinear form coincide with the coefficients of the characteristic polynomial of $h$. A similar property holds for the $t_k$ invariants as follows
\begin{proposition}
For each $1\leq i\leq n-1$, the higher cofactor transformations $t_i(h)$ coincide with the coefficients of the cofactor  characteristic polynomial $t_{n-1} (h-\lambda g)$, precisely we have
\begin{equation}\label{char-tn}
t_{n-1}(h-\lambda g)=\sum_{i=0}^{n-1} (-1)^{n-1-i} t_i(h)\lambda^{n-1-i}.
\end{equation}
\end{proposition}
\begin{proof}
 The binomial formula shows that
\begin{equation*}
\begin{split}
t_{n-1}(h-\lambda g)=& *\frac{(h-\lambda g)^{n-1}}{(n-1)!} \\
=&*\frac{1}{(n-1)!}\sum_{i=0}^{n-1}\binom{n-1}{i}h^i(-1)^{n-1-i}\lambda^{n-1-i}g^{n-1-i}\\
=& \frac{1}{(n-1)!}\sum_{i=0}^{n-1} (n-1)!(-1)^{n-1-i} \left(*\frac{g^{n-1-i} h^i}{(n-1-i)!i!}\right)\lambda^{n-1-i}\\
=&\sum_{i=0}^{n-1} (-1)^{n-1-i} t_i(h)\lambda^{n-1-i}.
\end{split}
\end{equation*}
Where $t_0=g$.
\end{proof}
\subsection{Laplace  expansions of the  determinant and generalizations}
 The following proposition provides a Laplace type expansion for all the $s_k$ invariants of an arbitrary  bilinear form.
\begin{proposition}
For each $k$, $0\leq k\leq n-1$, we have
\begin{equation}\label{laplaceexpansion}
(k+1)s_{k+1}(h)= \langle t_k(h), h\rangle . 
\end{equation}
Where $\langle .,.\rangle$ denotes the standard inner product of bilinear forms.
\end{proposition}
\begin{proof}
Using basic properties of the exterior product of double forms  and the generalized Hodge star operator, see section \ref{prem.}, it is straightforward that
\begin{equation*}
\langle t_k(h), h\rangle =*(\{*t_k(h)\}h)=*\left\{ \frac{g^{n-k-1}}{(n-k-1)!}
\frac{h^{k+1}}{k!}\right\}=(k+1)s_{k+1}(h).
\end{equation*}
\end{proof}
Let us now clarify its relation to the classical Laplace expansion (called also cofactor expansion) of the determinant. Remark that for $k=n-1$ we have
\begin{equation*}
ns_n=\langle t_{n-1}(h), h\rangle=\sum_{i,j=1}^nt_{n-1}(h)(e_i,e_j)h(e_i,e_j).
\end{equation*}
Where $\left(e_i\right)$ is an arbitrary orthonormal basis of $V$. Recall that by definition the factor $t_{n-1}(h)(e_i,e_j)$ is the usual $(ij)$-cofactor of the matrix $(h(e_i,e_j))$ and $s_n$ is its determinant. We therefore recover the classical Laplace expansion of the determinant.
Actually, Laplace expansion of the determinant is more refined. Precisely, it asserts that for any  $1\leq i\leq n$  we have
\begin{equation*}
s_n(h)=\sum_{j=1}^nt_{n-1}(h)(e_i,e_j)h(e_i,e_j).
\end{equation*}
In view of formula \ref{compo-product},  the later expansion can be written in the following form
\begin{equation}\label{Laplace-exp}
 \bigl(t_{n-1}(h)\bigr)^t\circ   h=s_n(h)g,
\end{equation}
or equivalently,
\[h^t\circ t_{n-1}(h)=s_n(h)g.\]
In other words the inverse of $h$ with respect to the composition  product $\circ$ is $\frac{1}{s_n(h)}  \bigl(t_{n-1}(h)\bigr)^t$. For a different  proof of the  last formula \ref{Laplace-exp} see Corollary I to proposition 7.4.1 in \cite{Greub-book}.\\
Following Greub \cite{Greub-book}, we use formula \ref{Laplace-exp} to get a useful formula for all the higher cofactor transformations $t_k(h)$ in terms of the composition product. Precisely we prove the following
\begin{proposition}\label{lemmatn}
Let  $1\leq i\leq n-1$  then we have the induction formula
\begin{equation}\label{ind-form}
t_i(h)=s_i(h)g-h^t\circ t_{i-1}(h).
\end{equation}
In particular, for $1\leq k\leq n-1$ we have
\begin{equation}\label{explicit-form}
t_k(h)=\sum_{r=0}^k(-1)^rs_{k-r}(h)(h^t)^{ \incircle}=s_k(h)g-s_{k-1}(h)h^t+s_{k-2}(h)h^t\circ h^t - ....
\end{equation}
\end{proposition}
\begin{proof}
Formula \ref{Laplace-exp} asserts that
\[(h-\lambda g)^t\circ t_{n-1}(h-\lambda g)=s_n(h-\lambda g)g.\]
Using the expansions \ref{char-tn}, \ref{char-sn} we get
\[(h^t-\lambda g)\circ \sum_{i=0}^{n-1}(-1)^{n-1-i}t_i(h)\lambda^{n-1-i}=\sum_{i=0}^n(-1)^{n-i}s_i(h)\lambda^{n-i}g.\]
A straightforward manipulation shows that
\[\sum_{i=1}^{n-1}\bigl(h^t\circ t_{i-1}(h)+t_i(h)-s_i(h)g\bigr)(-1)^{n-i}\lambda^{n-i}=0.\]
This completes the proof.
\end{proof}
\subsubsection{Further Laplace expansions}
Recall that the determinant of $h$ is determined by $h^n$. As the later expression can be written in several ways as a product $h^{n-r}h^r$ for each $r$, we therefore get different expansions for the determinant by blocks as follows:
\begin{equation*}
\begin{split}
&\det h =*\frac{h^n}{n!}=*\frac{h^{n-r}h^r}{n!}=\frac{(n-r)!r!}{n!}\langle *\frac{h^{n-r}}{(n-r)!},\frac{h^r}{r!}\rangle\\
&=\frac{1}{\binom{n}{r}} \sum_{{\scriptstyle i_1<i_2<...<i_r} \atop{\scriptstyle  j_1<j_2<...<j_r}}\epsilon(\rho)\epsilon(\sigma)\frac{h^r}{r!}\left(e_{i_1}, ...,e_{i_r},e_{j_1}, ..., e_{j_r}\right)\frac{h^{n-r}}{(n-r)!}\left(e_{i_{p+1}}, ..., e_{i_n}, e_{j_{p+1}}, ..., e_{j_n}\right).
\end{split}
\end{equation*}
Where $\{e_1,e_2,...,e_n\}$ is an orthonormal basis of $V$,  $\epsilon(\rho)$ and  $\epsilon(\sigma)$ are the signs of the permutations  $\rho=(i_1,...,i_n)$ and $\sigma=(j_1,...,j_n)$ of $(1,2,...,n)$.
Recall that $\frac{h^r}{r!}\left(e_{i_1}, ..., e_{i_r},e_{j_1},..., e_{j_r}\right)$  (resp. $\frac{h^{n-r}}{(n-r)!}\left(e_{i_{p+1}}, ..., e_{i_n}, e_{j_{p+1}}, ..., e_{j_n}\right)$ equals the determinant of the $r\times r$ sub-matrix $\left(h(e_{i_k},e_{j_l})\right)$ for $1\leq k,l\leq r$  (resp  the determinant of the $n-r\times n-r$ sub-matrix $\left(h(e_{i_k},e_{j_l})\right)$ for $p+1\leq k,l\leq n$). Note that the second sub-matrix is just the co-matrix of the first sub-matrix, that is the sub-matrix obtained from the ambient matrix $(h(e_i,e_j))$ of size $n$ after removing the rows $i_1,...,i_n$ and the columns $j_1,...,j_n$.  Let us mention here also that the original Laplace expansion is finer than the previous expansion, precisely it says that for any choice of $j_1,...,j_r$ we have
\begin{equation*}
\det h=\sum_{ i_1<i_2<...<i_r} \epsilon(\rho)\epsilon(\sigma) \frac{h^r}{r!}\left(e_{i_1}, ...,e_{i_r};e_{j_1}, ..., e_{j_r}\right)\frac{h^{n-r}}{(n-r)!}\left(e_{i_{p+1}}, ..., e_{i_n}; e_{j_{p+1}}, ..., e_{j_n}\right).
\end{equation*}
In view of formula \ref{compo-product},  the later expansion can be written in the following compact form
\begin{equation}
(\det h )\frac{g^r}{r!}=\bigl(*\frac{h^{n-r}}{(n-r)!}\bigr)^t\circ \frac{h^r}{r!} .
\end{equation}
This was first noticed in  \cite{Greub-book}, see  Corollary to Proposition 7.2.1.
\begin{remark}
One can write easily similar expansions for all the  lower $s_k$ invariants of $h$.  In fact,  the product  $g^{n-k}h^k$ can be written  in different ways as $g^ph^qg^{n-k-p}h^{k-q}$   for $0\leq q\leq k$ and $0\leq p\leq n-k$. Precisely we have, for $1\leq k\leq n$ and  for every $0\leq q\leq k$ and $0\leq p\leq n-k$, the following expansion
\begin{equation*}
\begin{split}
& k!(n-k)!s_k(h) =\ast g^{n-k}h^k =*\left( g^ph^qg^{n-k-p}h^{k-q}\right)
=\langle  g^ph^q, *g^{n-k-p}h^{k-q}\rangle=\\
&\sum_{{\scriptstyle   i_1<i_2<...<i_{p+q}} \atop{\scriptstyle   j_1<j_2<...<j_{p+q}}}\epsilon(\rho)\epsilon(\sigma) g^ph^q\left(e_{i_1}, ..., e_{i_{p+q}},e_{j_1},..., e_{j_{p+q}}\right) g^{n-k-p}h^{k-q} \left(e_{i_{p+q+1}}, ..., e_{i_n}, e_{j_{p+q+1}}, ..., e_{j_{n}}\right).
\end{split}
\end{equation*}
Where  as above $\{e_1,e_2,...,e_n\}$ is an orthonormal basis of $V$,  $\epsilon(\rho)$ and  $\epsilon(\sigma)$ are the signs of the permutations  $\rho=(i_1,...,i_n)$ and $\sigma=(j_1,...,j_n)$ of $(1,2,...,n)$. 
\end{remark}
\subsection{Girard-Newton  identities}
\begin{proposition}[Girard-Newton  identities]
For $0\leq k\leq n-1$, the trace of $t_k$ is given by 
\begin{equation}\label{newtid}
\ccc\, t_k(h)=(n-k)s_k(h). 
\end{equation}
Where $\ccc$ denotes the contraction map.
\end{proposition}
\begin{proof}
Using basic properties of the exterior product of double forms  and the  the generalized  Hodge star operator, see section \ref{prem.}, we immediately get
\begin{equation*}
\ccc\, t_k(h)=*g*t_k(h)=*\left\{ \frac{g^{n-k}}{(n-k-1)!}
\frac{h^k}{k!}\right\}= (n-k)s_k(h).
\end{equation*}
\end{proof}
In order to explain why the previous formula coincides with the classical Girard-Newton  identities,  we shall use proposition \ref{lemmatn}.  Let  for $1\leq i\leq n-1$, $p_i=\ccc\,  h^{ \iincircle}$. Proposition \ref{lemmatn} shows that 
\begin{equation*}
\ccc\,  t_r(h)=\sum_{i=0}^r(-1)^is_{r-i}(h)p_ i.
\end{equation*}
Therefore we can reformulate  the identity (\ref{newtid}) as
\begin{equation*}
\sum_{i=0}^r(-1)^is_{r-i}(h)p_ i=(n-r)s_r(h),
\end{equation*}
or
\begin{equation*}
rs_r(h)=\sum_{i=1}^r(-1)^{i+1}s_{r-i}(h)p_ i.
\end{equation*}
That are the celebrated classical Girard-Newton identities.
\begin{remark}[Terminology]
The transformations $t_r$ are famous in the literature as Newton's transformations. Up to the author's knowledge,  it was Reilly \cite{Reilly}  the first to call them as such (he treated only the case of diagonalizable matrices). He motivated this by the fact that they generate  the classical Newton identities as above. With reference to the above discussion, the terminology cofactor transformation or  characteristic transformation  is in the author's opinion  more appropriate.
\end{remark}
\subsection{Higher cofactor transformations  and Laplace expansions}
Let  $h$ be a bilinear form on the $n$-dimensional Euclidean vector space $(V,g)$. We define , for $0\leq q\leq n$ and $0\leq r\leq n-q$,  the $(r,q)$ cofactor (or the $(r,q)$ Newton transformation) of $h$ denoted $s_{(r,q)}(h)$ by
\begin{equation}
s_{(r,q)}(h)=\frac{1}{q!\bigl(n-q-r\bigr)!}*\bigl(g^{n-q-r}h^q\bigr).
\end{equation}
Note that $s_{(1,q)}(h)=t_q(h)$ is the cofactor  of order $q$ of $h$ as defined in subsection \ref{adjugate} and $s_{(0,q)}(h)=s_q(h)$ is the $s_q$ invariant of $h$.\\
The higher cofactors  $s_{(r,q)}(h)$ of $h$ satisfy similar properties like the usual  cofactor transformation of $h$ which was discussed above. We list some of them in the following theorem
\begin{theorem} For any  integers $r$ and $q$  such that $0\leq q\leq n$ and  $1\leq r\leq n-q$ we have
\begin{itemize}
\item  The  $(r,q)$ cofactors   $s_{(r,q)}(h)$ coincide with the coefficients of the characteristic polynomial $*\frac{(h-\lambda g)^{n-r}}{(n-r)!}$, precisely we have
\[*\frac{(h-\lambda g)^{n-r}}{(n-r)!}=\sum_{q=0}^{n-r}(-1)^{n-q-r}s_{(r,q)}(h)\lambda^{n-q-r}.\]
\item General Laplace's expansion:  
\[\frac{(q+r)!}{q!}s_{q+r}(h)=\langle s_{(r,q)}(h),h^r\rangle.\]
\item  General Newton's identity: 
\[ c \left( s_{(r,q)}(h)\right)=(n-q-r+1)s_{(r-1,q)},\]
\end{itemize}
\end{theorem}
\begin{proof}
The first statement is a direct consequence of the binomial formula. To prove the general Laplace's expansion we again use the properties of the exterior product of double forms to get a one line proof as follows
\[ \langle s_{(r,q)}(h),h^r\rangle=\ast \left(\frac{g^{n-q-r}h^q}{q!(n-q-r)!}h^r\right)=\ast \left(\frac{g^{n-q-r}h^{q+r}}{(q+r)!(n-q-r)!}\right)\frac{(q+r)!}{q!}=\frac{(q+r)!}{q!}s_{q+r}(h).\]
In the same way we prove the general  Newton's identity as follows
\begin{equation*}
\begin{split}
  c \left( s_{(r,q)}(h)\right)&=c\ast \left(\frac{g^{n-q-r}h^q}{q!(n-q-r)!}\right)=\ast g \left(\frac{g^{n-q-r}h^q}{q!(n-q-r)!}\right)\\
&=\ast \left(\frac{g^{n-q-r+1}h^q}{q!(n-q-r+1)}!\right)(n-q-r+1)=(n-q-r+1)s_{(r-1,q)}.
\end{split}
\end{equation*}
\end{proof}
\begin{remark}\begin{enumerate}
\item
If the bilinear form $h$ is diagonalizable, that is if there exists an orthonormal basis $(e_i)$ of $V$ such that $h(e_i,e_j)=\lambda_ig(e_i,e_j)$ for all $i,j$. We call the real numbers $\lambda_i$ the eigenvalues of $h$. Without loss of generality we assume that $\lambda_1\leq \lambda_2\leq ...\leq \lambda_n $.\\
It is not difficult to show that all the double forms $ s_{(r,q)}(h)$ with $r\geq 1$  are then also diagonalizable in the sense that
\[  s_{(r,q)}(h)\left(e_{i_1}, ..., e_{i_r},e_{j_1}, ...,e_{j_r}\right)=\lambda_{i_1i_2...i_r}\frac{g^r}{r!}\left(e_{i_1}, ..., e_{i_r},e_{j_1}, ...,e_{j_r}\right).\]
Where $i_1<i_2<...<i_r$, $j_1<j_2<...<j_r$  and the eigenvalues of $ s_{(r,q)}(h)$ are given by
\[ \lambda_{i_1i_2...i_r}=\sum_{{\scriptstyle j_1<j_2<...<j_q} \atop{\scriptstyle \{j_1,j_2,...,j_q\}\cap \{i_1,i_2,...,i_r\}=\phi }}\lambda_{j_1}\lambda_{j_2}...\lambda_{j_q}.\]
\item
 In some applications it is useful to find the determinant of the sum of two matrices or more generally the $s_k$ invariant of the sum. Using double forms formalism as above one can prove easily in one line the following identity for arbitrary bilinear forms $A$ and $B$ once seen as $(1,1)$ double forms and for $1\leq k\leq n$
\[ s_k(A+B)=\sum_{i=0}^k\frac{1}{(k-i)!}\langle s_{(k-i,i)}(A),B^{k-i}\rangle.\]
\item For a bilinear form $h$, denote by $\bar{h}$ the corresponding  linear operator $V\rightarrow V$. It turns out that  the linear operator corresponding to the double form $g^rh^k$ coincides with the $k$-linear extension of the operator $\bar{h}$ to the space  $\Lambda^{r+k}V$ in the sense of \cite{Sergei}.
\end{enumerate}
\end{remark}
\subsection{Jacobi's formula}
\begin{proposition}[A Jacobi's formula for the $s_k$ invariants]
Let $h=h(t)$ be a one parameter family of  bilinear forms on $V$ then 
\begin{equation}
\frac{d}{dt} s_k(h)=\langle t_{k-1}(h),\frac{dh}{dt}\rangle.
\end{equation}
In particular, for $k=n$ we recover the classical Jacobi's formula:
\begin{equation}
\frac{d}{dt} \det (h)=\langle t_{n-1}(h),\frac{dh}{dt}\rangle.
\end{equation}
\end{proposition}
\begin{proof}
Using basic properties of the exterior product of double forms and the generalized Hodge star operator, see \cite{Labbidoubleforms,Labbiminimal}, we  get
\begin{equation}
\begin{split}
\frac{d}{dt} s_k(h(t))=&\frac{d}{dt} \left(*\frac{g^{n-k}h^k(t)}{(n-k)!k!}\right)
=*\left( \frac{g^{n-k}kh^{k-1}}{(n-k)!k!}\frac{dh}{dt} \right)\\
=& *\left( *\left( * \frac{g^{n-k}h^{k-1}}{(n-k)!(k-1)!}\right)\frac{dh}{dt} \right)
= \langle t_{k-1}(h),\frac{dh}{dt}\rangle.
\end{split}
\end{equation}
\end{proof}
\begin{remark}
If one allows the inner product $g$ on $V$ to vary as well, say $g=g(t)$ ,  then at $t=0$ we have the following generalization of the  previous formula:
\begin{equation}
  \frac{d}{dt} s_k(h)=\langle t_{k-1}(h),\frac{dh}{dt}\rangle+\langle t_k-s_kg,\frac{dg}{dt}\rangle.
\end{equation}
The proof is  similar to the above one.
\end{remark}

\subsection{Cayley-Hamilton Theorem}\label{CaHa}
It is now time to give a sense to the top $t_k(h)$, that is $t_n(h)$, where $n$ is the dimension of the vector space $V$.  Recall that for $1\leq k\leq n-1$, formula \ref{explicit-form} asserts that
\begin{equation*}
t_k(h)=\sum_{r=0}^k(-1)^rs_{k-r}(h)(h^t)^{ \incircle}.
\end{equation*}

It is then natural to define $t_n(h)$ to be
\begin{equation}
t_n(h)=\sum_{r=0}^n(-1)^rs_{n-r}(h)(h^t)^{ \incircle}.
\end{equation}

\begin{proposition}[Cayley-Hamilton Theorem]  With the above notations we have
\[t_n(h)=0.\]
\end{proposition}
\begin{proof}
\begin{equation*}
\begin{split}
t_n(h)=& \sum_{r=0}^n(-1)^rs_{n-r}(h)(h^t)^{ \incircle}\\
=& s_n(h)g+ \sum_{r=1}^n(-1)^rs_{n-r}(h)(h^t)^{ \incircle}\\
=&  s_n(h)g- \sum_{r=1}^n(-1)^{r-1}s_{n-r}(h)(h^t)^{ \incircle}\\
=& s_n(h)g-h^t\circ t_{n-1}(h)=0.
\end{split}
\end{equation*}
\end{proof}

\subsubsection{Cayley-Hamilton theorem vs. Infinitisimal Gauss-Bonnet theorem}\label{infinitisimalGB}
Let $M$ be a compact smooth hypersurface of the Euclidean space of dimension $2n+1$. Denote by $B$ the second fundamental form of $M$ and by $s_k(B)$  its $s_k$ invariant.  For each $k$, $0\leq k\leq n$, the first variation of the integral $\int_Ms_{2k}(B){\rm dvol}$ is up to a multiplicative constant,  the integral scalar product $\langle t_{2k}(B), B\rangle$, where $t_{2k}(B)$ is the cofactor transformation of $B$ as above, see for instance  \cite{Labbiminimal}. The later result can be seen as an integral Jacobi's formula.  By Cayley-Hamilton theorem $t_{2n}(B)=0$ and therefore the integral $\int_Ms_{2n}(B){\rm dvol}$ does not depend on the geometry of the hypersurface. In fact, the previous integral is up to a multiplicative constant the Euler-Poincar\'e characteristic  by the Gauss-Bonnet theorem. \\
In this sense, Cayley-Hamilton theorem is  indeed an infinitisimal Gauss-Bonnet theorem.

\subsubsection{Further algebraic identities of Cayley-Hamilton type}
In order to simplify the exposition, we  assume in this subsection   that $h$ is a symmetric bilinear form. We are going first to give an alternative proof for the Cayley-Hamilton theorem.\\
Recall that for $1\leq k\leq n-1$ we have, see \cite{Labbidoubleforms,Labbiminimal}
\begin{equation*}
t_k(h)=\frac{1}{k!\bigl(n-1-k\bigr)!}*\bigl(g^{n-1-k}h^k\bigr)=s_k(h)g-\frac{1}{(k-1)!}c^{k-1}h^k.
\end{equation*}
It is then natural to define $t_n(h)$ to be
\begin{equation*}
t_n(h)=s_n(h)g-\frac{1}{(n-1)!}c^{n-1}h^n.
\end{equation*}
Recall that $*h^n=n!s_n(h)$, and therefore $h^n=n!s_n(h)*1=s_n(h)g^n$.  Consequently we have
\begin{equation*}
t_n(h)=s_n(h)g-s_n(h)\frac{1}{(n-1)!}c^{n-1}g^n=s_n(h)g-s_n(h)g=0.
\end{equation*}
Next we are going to prove similar results for the higher cofactors $s_{(r,q)}(h)$. Formula (15) of \cite{Labbidoubleforms} provides the following expansion for $0\leq r\leq n$ and $1\leq q\leq n-r$:
\begin{equation}\label{higherq}
s_{(r,q)}(h)=\sum_{i=\max \{0,q-r\}}^q\frac{(-1)^{i+q}}{i!q!(i+r-q)!}g^{i+r-q}c^ih^q.\end{equation}
This new form of $ s_{(r,q)}(h)$ allows to extend its definition to the higher values of $q$, namely for $q$ equal to $n-r+1, ...,n$.  For instance in the top case $q=n\geq 2$ we define
\begin{equation}
s_{(r,n)}(h)=\sum_{i= n-r}^n\frac{(-1)^{i+n}}{i!n!(i+r-n)!}g^{i+r-n}c^ih^n.
\end{equation}
Recall that $h^n=s_n(h)g^n$ and therefore $c^ih^n=s_n(h)\frac{i!n!}{(n-i)!}g^{n-i}$, consequently we have for $r\geq 1$
\[s_{(r,n)}(h)=\sum_{i= n-r}^n \frac{(-1)^{i+n}} {(n-i)!(i+r-n)!} g^r=\frac{(-1)^r }{r!} \left( \sum_{j=0}^r (-1)^j\binom{r}{j}\right) g^r=0.\]
We have therefore proved that
\begin{equation}
s_{(r,n)}(h)=0\,\,\, {\rm for \,\, all}\,\, 1\leq r\leq n.
\end{equation}
For $r=1$ we recover $t_n(h)=0$ that is the usual Cayley-Hamilton theorem for $h$. 
The next case is when $q=n-1\geq 2$ and $1\leq r\leq n-1$, here we set
\begin{equation}
s_{(r,n-1)}(h)=\sum_{i= n-r-1}^{n-1}\frac{(-1)^{i+n-1}}{i!(n-1)!(i+r-n+1)!}g^{i+r-n+1}c^ih^{n-1}.
\end{equation}
 Next we are going to show that  
\begin{equation}
s_{(r,n-1)}(h)=0\,\,\, {\rm for \,\, all}\,\, 2\leq r\leq n-1.
\end{equation}
In order to prove the above identities, first remark that $h^{n-1}$ is a $(n-1,n-1)$ double form on an $n$ dimensionnal vector space, then using proposition 2.1 of \cite{LabbipqEinstein} we can write $h^{n-1}=g^{n-2}k$ for some $(1,1)$ double form $k$ on $V$. Consequently, using some identities from \cite{Labbidoubleforms}  we get for $i\leq n-2$ the following
\begin{equation*}
\begin{split}
c^i(h^{n-1})=& c^i(g^{n-2}k)=\ast g^i\ast \frac{g^{n-2}k}{(n-2)!}(n-2)!\\
=& \ast g^i\left(-k+gck\right)(n-2)!=\left(-\ast g^ik+\ast g^{i+1}ck\right)(n-2)!\\
=& -i!(n-2)!\left( -\frac{g^{n-i-2}k}{(n-i-2)!}+\frac{g^{n-i-1}ck}{(n-i-1)!}\right)+\frac{(i+1)!)n-2)!}{(n-i-1)!}g^{n-i-1}ck\\
=& \frac{(n-2)!i!}{(n-i-2)!}g^{n-i-2}\left( k+\frac{i(ck)}{n-i-1}g\right).\\
\end{split}
\end{equation*}
For $i=n-1$, we get
\[ c^{n-1}h^n=\left( (n-1)!\right)^2 ck.\]
Consequently the formula above defining $s_{(r,n-1)}(h)$ takes the form
\begin{equation*} 
\begin{split}
& s_{(r,n-1)}(h)= \frac{ck}{r!}g^r+\sum_{i=n-1-r}^{n-2}\frac{(-1)^{i+n-1}}{(n-1)(i+r-n+1)!(n-i-2)!}g^{r-1}\left(k+\frac{i(ck)}{n-i-1}g\right)\\
=&\left(\sum_{i=n-1-r}^{n-2}\frac{(-1)^{i+n-1}}{(i+r-n+1)!(n-i-2)!}\right)\frac{g^{r-1}k}{n-1}+\left(\sum_{i=n-1-r}^{n-1}\frac{(-1)^{i+n-1}i}{(i+r-n+1)!(n-i-1)!}\right)\frac{g^rck}{n-1}.
\end{split}
\end{equation*}
Changing the index of both sums to $j=i-n+1+r$ we immediately obtain
\[s_{(r,n-1)}(h)=\left(\sum_{j=0}^{r-1}\frac{(-1)^j}{j!(r-j-1)!} \right)\frac{(-1)^r g^{r-1}k}{n-1}+\left(\sum_{j=0}^r\frac{(-1)^j(j+n-1-r)}{j!(r-j)!}\right)\frac{ (-1)^rg^rck}{n-1}.\]
It is then easy to check that the previous two sums are both zero for $r\geq 2$.\\
In the same way we define $s_{(r,n-i)}(h)$ using formula (\ref{higherq}), one can prove similarly, as in the  cases where $i=0$ and $i=1$ above, the following general result
\begin{theorem}[A general Cayley-Hamilton theorem]\label{CHGtheorem}
For  $1\leq i+1\leq r\leq n-i$ we have
\[ s_{(r,n-i)}(h)=0.\]
\end{theorem}
Finally, let us mention that it would be interesting to reformulate in a nice way the previous theorem in terms of the composition product.

\section{Euclidean invariants for Symmetric $(2,2)$ Double Forms}\label{second-section}
In this section we are  going to generalize the previous results to symmetric $(2,2)$ double forms that satisfy the first Bianchi identity. Recall that  a $(2,2)$ symmetric double form is a multilinear form with four arguments that is skew symmetric with respect to the interchange of the first two arguments or the last two, and it is symmetric if we interchage the first two arguments with the last two.  The Riemann curvature tensor is a typical example.\\
During this section  $R$  denotes a  symmetric  $(2,2)$ double form on the $n$-dimensional Euclidean vector space  $(V,g)$ that  satisfies the first Bianchi identity. 
\subsection{The $h_{2k}$ invariants vs $s_k$ invariants}
 For each $k$, $0\leq 2k\leq \dim V=n$, we define (by analogy to the $s_k$ invariants of the previous section)   the $h_{2k}$ invariant of  $R$ to be 
\begin{equation}
h_{2k}(R)=\frac{1}{(n-2k)!}*\bigl( g^{n-2k}R^k\bigr).
\end{equation}
In particular $h_0=1$ and $h_n=*R^k$ in case $n=2k$. \\
In the case where $R$ is the Riemann curvature tensor of a Riemannian manifold, $h_{2k}(R)$ is know as the $2k$-th Gauss-Bonnet curvature.
\begin{remark}
Suppose $n=2k$ is even and define the $h_n$-characteristic polynomial of $R$ to be 
$h_n\left(R-\lambda \frac{g^2}{2}\right)$. A direct computation shows that
\begin{equation*}
\begin{split}
h_n\left( R-\lambda \frac{g^2}{2}\right)=& \ast \left( R-\lambda \frac{g^2}{2}\right)^k\\
=& \ast\sum_{i=0}^k\binom{k}{i}R^i\frac{(-1)^{k-i}}{2^{k-i}}\lambda^{k-i}g^{2k-2i}=\sum_{i=0}^k\binom{k}{i}\frac{(-1)^{k-i}}{2^{k-i}}\lambda^{k-i}\ast \left(g^{2k-2i}R^i\right)\\
=& \sum_{i=0}^k\binom{k}{i}\frac{(-1)^{k-i}}{2^{k-i}}(2k-2i)!h_{2i}(R)\lambda^{k-i}.
\end{split}
\end{equation*}
The $h_n$-characteristic polynomilal of $R$ is therefore a polynomial of degree $k$ in $\lambda$ and its coefficients are all the $h_{2i}(R)$ invariants of $R$. It would be interesting to see whether for each $k$, the  polynomial $h_{2k}(R)$, which is homogeneous of degree $2k$  and defined on the space of  symmetric $(2,2)$ double forms,  is a  hyperbolic  polynomial with respect to $\frac{g^2}{2}$ in the sens of G\r{a}rding \cite{Garding}.
\end{remark}
\subsection{Cofactor transformations  of $(2,2)$ double forms}
 Using the same procedure of cofactors as in the previous section we obtain several  Newton  transformations  of $R$.  Let us here examine the following
\begin{equation}
N_{2k}(R)=*\frac{g^{n-2k-2}R^k}{(n-2k-2)!}\,\, {\rm and}\,\, T_{2k}(R)=*\frac{g^{n-2k-1}R^k}{(n-2k-1)!}.
\end{equation}
Note that $N_{2k}(R)$ is defined for $2\leq 2k\leq n-2$ and it  is a  $(2,2)$ symmetric double form on $V$  like $R$ that satisfies the first Bianchi identity. 
On the other hand $T_{2k}(R)$ is defined for $2\leq 2k\leq n-1$ and it  is a  symmetric bilinear form on $V$.\\
Theorem 4.1  of \cite{Labbidoubleforms} provides  explicit useful formulas for all the $N_{2k}(R)$ and $T_{2k}(R)$ as follows:
\begin{equation}\label{explicitNk}
\begin{split}
N_{2k}(R)&=\frac{c^{2k-2}R^k}{(2k-2)!}-\frac{c^{2k-1}R^k}{(2k-1)!}g+\frac{c^{2k}R^k}{2(2k)!}g^2,\\
T_{2k}(R)&=\frac{c^{2k}R^k}{(2k)!}g -\frac{c^{2k-1}R^k}{(2k-1)!}.
\end{split}
\end{equation}
In particular, $T_2(R)=\frac{c^{2}R}{2}g -cR$ is the celebrated Einstein Tensor. The higher $T_{2k}(R)$ are called Einstein-Lovelock tensors, see \cite{Labbivariation}.

\subsection{Laplace expansion of the $h_{2k}$ invariants and Avez formula}
\begin{theorem}[Laplace expansion of the $h_{2k}$ invariants]
For $4\leq 2k+2\leq n$ we have
\begin{equation}\label{GBnewformula}
h_{2k+2}(R)=\langle N_{2k}(R),R\rangle .
\end{equation}
\end{theorem}

\begin{proof}
Using the results of  \cite{Labbidoubleforms} one immediately has
\[
\langle N_{2k}(R),R \rangle=*\left\{ \frac{g^{n-2k-2}R^{k+1}}{(n-2k-2)!}\right\}=\frac{c^{2k+2}R^{k+1}}{(2k+2)!}.
\]
This completes the proof.
\end{proof}

As a consequence we recover  Avez's formula for the second Gauss-Bonnet curvature  as follows:
\begin{corollary}[Avez's Formula]
For $n\geq 4$, the second Gauss-Bonnet curvature equals
\begin{equation*}
h_4(R)=|R|^2-|cR|^2+\frac{1}{4}|c^2R|^2.
\end{equation*}
\end{corollary}
\begin{proof} A direct application of formula \ref{explicitNk} shows that
\[ N_2(R)=R- (cR)g+\frac{c^{2}R}{4}g^2.
\]
Consequently,
\[ h_4(R)=\langle N_2(R),R\rangle=\langle R,R\rangle -\langle (cR)g,  R\rangle+ \langle \frac{c^{2}R}{4}g^2,R\rangle.\]
To complete the proof just recall that the contraction map $c$ is the adjoint of the multiplication map by the $g$.
\end{proof}
In the same way one can prove easily the following generalization of Avez's formula, see \cite{Labbiyamabe},
\begin{corollary}\label{corollaryh2k+2}
For $4\leq 2k+2\leq n$, the $(2k+2)$-th Gauss-Bonnet curvature is determined by the last three contractions of $R^k$  as follows:
\begin{equation*}
h_{2k+2}=\langle \frac{c^{2k-2}R^k}{(2k-2)!},R\rangle-\langle \frac{c^{2k-1}R^k}{(2k-1)!},cR\rangle +h_{2k}h_2.
\end{equation*}
\end{corollary}

\subsection{Girard-Newton identities}
\begin{theorem} Let $c$ denotes the contraction map then we have
\begin{equation*}
cN_{2k}(R)=(n-2k-1)T_{2k}\, \, \, {\rm and}\,\,\, c T_{2k}(R)=(n-2k)h_{2k}.
\end{equation*} 
\end{theorem}
\begin{proof}
Using the identity $c*=*g$,  one easily  gets the desired formulas as follows:
\begin{equation*}
cN_{2k}(R)=c*\frac{g^{n-2k-2}R^k}{(n-2k-2)!}=*\frac{g^{n-2k-1}R^k}{(n-2k-1)!}\frac{(n-2k-1)!}{(n-2k-2)!}=(n-2k-1)T_{2k}.
\end{equation*}
Similarly,
\begin{equation*}
cT_{2k}(R)=c*\frac{g^{n-2k-1}R^k}{(n-2k-1)!}=*\frac{g^{n-2k}R^k}{(n-2k)!}\frac{(n-2k)!}{(n-2k-1)!}=(n-2k)h_{2k}.
\end{equation*}

\end{proof}

\subsection{ Algebraic identities for $(2,2)$ double forms}
\subsubsection{The case of even dimensions}
 Suppose the dimension of the vector spave $V$ is even $n=2k$. We shall now give a sense the top $T_{2k}(R)$. Using formula \ref{explicitNk} we naturally set
\begin{equation}\label{Nn}
T_{n}(R)=\frac{c^{n}R^k}{n!}g -\frac{c^{n-1}R^k}{(n-1)!}.
\end{equation}

\begin{proposition}\label{PropositionTnNn}
Let $R$ be a symmetric $(2,2)$ double form satisfying the first Bianchi identity on an Euclidean space of even dimension $n$  then
\begin{equation}
T_n(R)=0.
\end{equation}
\end{proposition}

\begin{proof}
Note first that if $n=2k$ then $h_n=*R^k$ and therefore $R^k=*h_n$. Where $*$ is the  Hodge star operator acting on double forms. For any $0\leq r\leq n$ we have
\[ c^rR^k=c^r *h_n=* g^rh_n=r!\frac{g^{n-r}}{(n-r)!}h_n.\]
That is 
\[ \frac{c^rR^k}{r!}=\frac{h_n}{(n-r)!}g^{n-r}.\]
Next using the definition of $T_n$ above we easily get that
\[ T_n=h_ng-h_ng=0.\]
\end{proof}
\begin{remark}
In this case where $n=2k$ one could using formula (\ref{explicitNk}) define $N_n(R)$ as well by setting 
\[N_{n}(R)=\frac{c^{n-2}R^k}{(n-2)!}-\frac{c^{n-1}R^k}{(n-1)!}g+\frac{c^{n}R^k}{2(n!)}g^2. \]
A direct adaptation of the previous proof shows that  we   have the algebraic identity $N_n(R)=0$.
\end{remark}

\subsubsection{The case of odd dimensions}
 Suppose now the dimension of the vector spave $V$ is odd say $n=2k+1$. We shall now give a sense the top  $N_{n-1}(R)$. Note that the top $T_{n-1}(R)$ is well defined and need not vanish in general.  Using formula \ref{explicitNk} we naturally set for $n\geq 3$
\begin{equation}\label{Nnodd}
N_{n-1}(R)=\frac{c^{n-3}R^k}{(n-3)!}-\frac{c^{n-2}R^k}{(n-2)!}g+\frac{c^{n-1}R^k}{2(n-1)!}g^2.
\end{equation}

\begin{theorem}\label{theoremNn-1}
Let $R$ be a symmetric $(2,2)$ double form satisfying the first Bianchi identity on an Euclidean space of odd dimension $n\geq 3$  then
\begin{equation}
N_{n-1}(R)=0.
\end{equation}
\end{theorem}

\begin{proof}
Let $n=2k+1\geq 3$, note that $R^k$ is a $(n-1,n-1)$ double form on an $n$ dimensionnal vector space, then using proposition 2.1 of \cite{LabbipqEinstein} we can write $R^k=g^{n-2}D$ for some $(1,1)$ double form $D$ on $V$. Consequently, using some identities from \cite{Labbidoubleforms}  we get  the following
\begin{equation*}
\begin{split}
c^{2k-2}R^k=& c^{n-3}\left(g^{n-2}D\right)=\ast g^{n-3}\ast g^{n-2}D\\
=& \ast g^{n-3}(n-2)!(-D+gcD)=(n-2)!\left( -\ast g^{n-3}D+\ast g^{n-2}cD\right)\\
=& (n-2)!(n-3)!\left( gD+\frac{(n-3)cD}{2}g^2\right).\\
\end{split}
\end{equation*}
After contracting the previous identity twice we get
\[c^{2k-1}R^k= (n-2)!(n-2)!\left( D+(n-2)(cD)g \right), \,\, {\rm and}\,\,  c^{2k}R^k=(n-1)!(n-1)!cD.\]
Consequently we have
\begin{equation*}
\begin{split}
N_{n-1}(R)=& \frac{c^{n-3}R^k}{(n-3)!}-\frac{c^{n-2}R^k}{(n-2)!}g+\frac{c^{n-1}R^k}{2(n-1)!}g^2\\
=& (n-2)!\left( D+\frac{(n-3)cD}{2}g-D-(n-2)cDg+(n-1)\frac{cD}{2}g\right)g=0.
\end{split}
\end{equation*}
\end{proof}
\begin{remark}
In dimension $n=3$ the previous theorem read
\[ N_{2}(R)=R-(cR)g+\frac{c^2R}{4}g^2=0.\]
In the context of Riemannian geometry where $R$ represents the Riemann curvature tensor the previous identity is equivalent to the vanishing of the Weyl tensor in $3$ dimensions, in fact in this dimension $ N_{2}(R)$ coincides with  the Weyl tensor.
\end{remark}

\subsubsection{Algebraic scalar identities  for $(2,2)$ double forms}
Suppose the dimension $n$ of our vector space $V$ is odd, say $n=2k+1$ and as above $R$ is a symmetric  $(2,2)$ double form that  satisfies the first Bianchi identity. . Corollary \ref{corollaryh2k+2} allows one to define $h_{n+1}(R)=h_{2k+2}(R)$, precisely we set
\begin{equation*}
h_{2k+2}(R)=\langle \frac{c^{2k-2}R^k}{(2k-2)!},R\rangle-\langle \frac{c^{2k-1}R^k}{(2k-1)!},cR\rangle +\langle \frac{c^{2k}R^k}{(2k)!},\frac{c^2R}{2}\rangle.
\end{equation*}
We are going to show that $h_{2k+2}(R)$ as defined by the previous equation is zero.\\
We proceed as in the proof of Theorem \ref{theoremNn-1} and using the same notations  of that proof we have
\begin{equation*}
\begin{split}
\langle \frac{c^{2k-2}R^k}{(2k-2)!},R\rangle &=(n-2)!\langle D+\frac{(n-3)cD}{2}g,R\rangle,\\
-\langle \frac{c^{2k-1}R^k}{(2k-1)!},cR\rangle &=-(n-2)!\langle D+(n-2)cDg,cR\rangle,\\
\langle \frac{c^{2k}R^k}{(2k)!},\frac{c^2R}{2}\rangle &=(n-1)!c D\frac{c^2R}{2}.\\
\end{split}
\end{equation*}
Taking the sum of the above three equation we immediately prove the vanishing of $h_{2k+2}(R)$. Thus we have proved the following scalar identities
\begin{proposition}
Let $R$ be a symmetric $(2,2)$ double form satisfying the first Bianchi identity on an Euclidean vector space of odd dimension $n=2k+1\geq 3$  then
\[ \langle \frac{c^{n-3}R^k}{(n-3)!},R\rangle-\langle \frac{c^{n-2}R^k}{(n-2)!},cR\rangle +\langle \frac{c^{n-1}R^k}{(n-1)!},\frac{c^2R}{2}\rangle=0.\]
In particular, for $n=3$ we have
\[ \langle R,R\rangle-\langle cR,cR\rangle + \frac{1}{4}\left( c^2R\right)^2=0.\]
\end{proposition}
\begin{remark}
It the context of Riemannian geometry, where $R$ is the Riemann curvature tensor (seen as a $(2,2)$ double form), the previous  scalar curvature identities coincide with Gilkey-Park-Sekigawa  universal curvature identities  \cite{GPS} which are shown to be unique. Also the identities of Proposition \ref{PropositionTnNn} coincide with the  symmetric $2$-form valued universal curvature identities of \cite{GPS} where they are also be shown to be unique. The higher algebraic  identities, that are under study  here in this paper,  can be seen then as symmetric double form valued universal curvature identities in the frame of Riemannian geometry.

\end{remark}

\subsubsection{Higher algebraic identities for $(2,2)$ double forms}
Let $k\geq 1$ and $0\leq r\leq n-2k$ and let $R$ as above be a  symmetric  $(2,2)$ double form on the $n$-dimensional Euclidean vector space  $(V,g)$ that  satisfies the first Bianchi identity. We define the $(r,2k)$-cofactor transformation ot $R$, denoted $h_{(r,2k)}(R)$,  by the following formula
\begin{equation}
h_{(r,2k)}(R)=\frac{1}{(n-2k-r)!}\ast \bigl( g^{n-2k-r}R^k\bigr).
\end{equation}
Note that $h_{(r,2k)}(R)$ for  $r=0$ (resp. $r=1$ , $r=2$) coincides with $h_{2k}(R)$  (resp. $T_{2k}(R)$ ,  $N_{2k}(R)$ ).\\
Theorem 4.1  of \cite{Labbidoubleforms} shows that
\begin{equation}\label{hrkexpansion}
 h_{(r,2k)}(R)=\sum_{i=\max\{0,2k-r\}}^{2k}\frac{(-1)^{i}}{i!(r-2k+i)!}g^{r-2k+i}c^iR^k.
\end{equation}
This last formula (\ref{hrkexpansion}) allows us to define $h_{(r,2k)}(R)$ for higher $r$'s that is for  $r>n-2k$. The following theorem which provides general identities and generalize Proposition \ref{PropositionTnNn} and  Theorem \ref{theoremNn-1} can be proved in the same way 
\begin{theorem}\label{theorem-even-odd}
Let $R$ be a symmetric $(2,2)$ double form satisfying the first Bianchi identity on an Euclidean vector space of  dimension $n$.
\begin{enumerate}
\item If $n=2k$ is even then
\[h_{(r,n-2i)}(R)=0\,\,\, {\rm for}\,\,\, 2i+1\leq r\leq n-2i.\]
\item If $n=2k+1$ is odd then
\[h_{(r,n-2i-1)}(R)=0\,\,\, {\rm for}\,\,\, 2i+2\leq r\leq n-2i-1.\]
\end{enumerate}
\end{theorem}
Remark that we recover Proposition \ref{PropositionTnNn} for $n=2k,r=1, i=0$ and Theorem \ref{theoremNn-1} is obtained for $n=2k+1,r=2, i=0.$
\subsection{Jacobi's formula for double forms}
\begin{proposition}[Jacobi's formula]
Let $R=R(t)$ be a one parameter family of  $(2,2)$ double forms then 
\begin{equation}
\frac{d}{dt} h_{2k}(R)=\langle k N_{2k-2}(R),\frac{dR}{dt}\rangle.
\end{equation}

\end{proposition}
\begin{proof}

\begin{equation}
\begin{split}
\frac{d}{dt} h_{2k}(R)=&\frac{d}{dt} \left(*\frac{g^{n-2k}R^k(t)}{(n-2k)!}\right)
=*\left( \frac{g^{n-2k}kR^{k-1}}{(n-2k)!}\frac{dR}{dt} \right)\\
=& *\left( *\left( * \frac{g^{n-2k}kR^{k-1}}{(n-2k)!}\right)\frac{dR}{dt} \right)
= \langle k N_{k-1}(R),\frac{dR}{dt}\rangle.
\end{split}
\end{equation}
\end{proof}
\begin{remark}
If one allows the scalar product to vary as well say $g=g(t)$ the previous formula takes the following form at $t=0$:
\begin{equation}
\frac{d}{dt}  h_{2k}(R)= \langle k N_{2k-2}(R),\frac{dR}{dt}\rangle+\langle T_{2k}-h_{2k}g,\frac{dg}{dt}\rangle.
\end{equation}
The proof is  similar to the above one.
\end{remark}
\subsection{Algebraic identities vs. infinitisimal Gauss-Bonnet theorem}
Let $(M,g)$ be a compact Riemannian manifold of dimension $n=2k$. Denote by $R$ its Riemann curvature tensor seen here as a $(2,2)$ double form and, and let  $h_{2r}(R)$ be the corresponding Gauss-Bonnet curvatures as above. For each $r$, $0\leq r\leq n$, the gradient of the Riemannian functional  $\int_Mh_{2r}(R){\rm dvol}$ at $g$, once restricted to metrics of unit volume,  is equal to    $T_{2r}(R)$, where $T_{2r}(R)$ is the $h_{2r}(R)$ cofactor of $R$ as above, it is known in geometry as the Einstein-Lovelock tensor, see \cite{Labbivariation}. The later result can be seen as an integral Jacobi's formula. Consequently, the algebraic identity  $T_{n}(R)=0$  shows that the integral $\int_Mh_{n}(R){\rm dvol}$ does not depend on the metric $g$ of $M$. In fact, the previous integral is up to a multiplicative constant the Euler-Poincar\'e characteristic  by the Gauss-Bonnet theorem. \\
Here again, as in the situation of section  \ref{infinitisimalGB},  the linearized version of the Gauss-Bonnet theorem is an algebraic identity for $(2,2)$ double forms.

\section{Higher cofactor transformations and applications}\label{third-section}
We shall now in this section generalize the previous results to higher
symmetric $(p,p)$-double forms.\\
 Let  $\omega$ be a symmetric $(p,p)$-double form satisfying the first Bianchi identity, we define its cofactor transformation of order $(r,pq)$ to be
\begin{equation}
h_{(r,pq)}(\omega)=*\frac{g^{n-pq-r}\omega^q}{(n-pq-r)!}.
\end{equation}
Where $0\leq r\leq n-pq$.
The result is a symmetric $(r,r)$-double form. We remark that for $\omega=h$  a  $(1,1)$ symmetric double form that is a symmetric bilinear form we recover the invariants of section \ref{first-section}  as follows
\[h_{(0,q)}(h)=q!s_q(h), \,\,  h_{(1,q)}(h)=q!t_q(h) \,\, {\rm and} \,\, h_{(r,q)}(h)=q!s_{(r,q)}(h).\]
Furthermore, for $\omega=R$ a symmetric $(2,2)$ double form, we recover the invariants of section 2 as follows
\[h_{(0,2q)}(R)=h_{2q}(R),\,\,  h_{(1,2q)}(R)=T_{2q}(R) \,\, {\rm and}\,\,  h_{(2,2q)}(R)=N_{2q}(R).\]
 One can generalize without difficulties the results of the previous sections to this general setting.  First using Theorem 4.1 of \cite{Labbidoubleforms} one easily gets
\begin{equation}\label{Nrqexpansion}
 h_{(r,pq)}(\omega)=\sum_{i=\max\{0,pq-r\}}^{pq}\frac{(-1)^{i+pq}}{i!(r-pq+i)!}g^{r-pq+i}c^i\omega^q.
\end{equation}
As a first result we have the following  Laplace type expansion:
\begin{theorem}
 Let  $\omega$ be a symmetric $(p,p)$-double form satisfying the first Bianchi identity on an $n$-dimensional Euclidean vector space $V$, Let $q$ be a positive integer such that $n\geq 2pq$ then
\begin{equation}
\frac{c^{2pq}(\omega^{2q})}{(2pq)!}=\langle h_{(pq,pq)}(\omega),\omega^q\rangle=\sum_{r=0}^{pq}\frac{(-1)^{r+pq}}{(r!)^2}\langle c^r\omega^q, c^r\omega^q\rangle.
\end{equation}
\end{theorem}

\begin{proof}
From one hand we have
\begin{equation}
\langle h_{(pq,pq)}(\omega),\omega^q\rangle= \langle  *\frac{g^{n-2pq}\omega^q}{(n-2pq)!},\omega^q\rangle=*\left( \frac{g^{n-2pq}\omega^{2q}}{(n-2pq)!}\right)=\frac{c^{2pq}(\omega^{2q})}{(2pq)!}.
\end{equation}
On the other hand we have
\begin{equation}
*\left(\frac{g^{n-2pq}\omega^q}{(n-2pq)!}\right)=\sum_{r=0}^{pq} \frac{(-1)^{r+pq}}{r!} \frac{g^rc^r\omega^q}{r!}.
\end{equation}
Consequently it is straightforward that
\begin{equation}
\frac{c^{2pq}(\omega^{2q})}{(2pq)!}=\sum_{r=0}^{pq}\frac{(-1)^{r+pq}}{(r!)^2}\langle g^r c^r\omega^q, \omega^q\rangle=\sum_{r=0}^{pq}\frac{(-1)^{r+pq}}{(r!)^2}\langle c^r\omega^q, c^r\omega^q\rangle.
\end{equation}
\end{proof}
Taking $\omega=R$ a $(2,2)$ double form  we get
\begin{corollary}[General Avez Formula]
Let  $R$ be a symmetric $(2,2)$-double form satisfying the first Bianchi identity on an $n$-dimensional Euclidean vector space $V$, Let $q$ be a positive integer such that $n\geq 4q$ then
\begin{equation}
h_{4q}(R)=\sum_{r=0}^{2q}\frac{(-1)^{r}}{(r!)^2}\langle c^rR^q, c^rR^q\rangle.
\end{equation}
In particular,
\begin{equation}
h_4=|R|^2-|cR|^2+\frac{1}{4}|c^2R|^2 \, {\rm and} \, h_8=|R^2|^2-|cR^2|^2+\frac{1}{4}|c^2R^2|^2-\frac{1}{36}|c^3R^2|^2+\frac{1}{576}|c^4R^2|^2.
\end{equation}

\end{corollary}
\begin{remark}
The previous corollary improves a similar result of \cite{Labbidoubleforms} obtained for the case $n=4q$.
\end{remark}
\begin{corollary}
Let  $h$ be a symmetric bilinear form  on an $n$-dimensional Euclidean vector space $V$, Let $q$ be a positive integer such that $n\geq 2q$ then
\begin{equation}
s_{2q}(h)=\sum_{r=0}^{q}\frac{(-1)^{r+q}}{(r!)^2}\langle c^rh^q, c^rh^q\rangle.
\end{equation}
In particular,
\begin{equation}
s_2(h)=-|h|^2+|ch|^2 \, {\rm and} \, s_4(h)=|h^2|^2-|ch^2|^2+\frac{1}{4}|c^2h^2|^2.
\end{equation}
\end{corollary}
The following proposition can be seen as a generalization of the classical Newton identities 
\begin{proposition}
For $1\leq r\leq n-pq$ we have
\[ c\left(h_{(r,pq)}(\omega)\right)=(n-pq-r+1)h_{(r-1,pq)}(\omega)\].
\end{proposition}
\begin{proof}
Using the identity $c*=*g$,  one easily  gets the desired formulas as follows:
\begin{equation*}
c\left(h_{(r,pq)}(\omega)\right)=c*\frac{g^{n-pq-r}\omega^k}{(n-pq-r)!}=*\frac{g^{n-pq-r+1}R^k}{(n-pq-r)!}=(n-pq-r+1)h_{(r-1,pq)}(\omega).
\end{equation*}
\end{proof}
We finish this section by establishing higher algebraic identities.
First note that formula (\ref{Nrqexpansion}) allows  to define $h_{(r,2k)}(\omega)$ for higher $r$'s that is in the cases where  $r>n-pk$. The following theorem provides general algebraic  identities and generalize Theorems \ref{theorem-even-odd} and \ref{theoremNn-1}  and  Proposition \ref{PropositionTnNn}. It can be proved by imitating the proof of Theorem  \ref{theoremNn-1}.
\begin{theorem}
Let $\omega$ be a symmetric $(p,p)$ double form satisfying the first Bianchi identity on an Euclidean vector space of  dimension $n$. Then 
\[h_{(r,pk-pi)}(\omega)=0,\,\,  {\rm for}\,\,\, n-pk+pi+1\leq r\leq pk-pi.\]
In particular, we have
\begin{enumerate}
\item If $n=pk$ is is a multiple of $p$ then
\[h_{(r,n-pi)}(\omega)=0\,\,\, {\rm for}\,\,\, pi+1\leq r\leq n-pi.\]
\item If $n=pk+1$ is then
\[h_{(r,n-pi-1)}(\omega)=0\,\,\, {\rm for}\,\,\, pi+2\leq r\leq n-pi-1.\]
\end{enumerate}
\end{theorem}
Remark that we recover the results of  Theorem \ref{theorem-even-odd}  for $p=2$,  Proposition \ref{PropositionTnNn} for $n=2k, r=1, i=0$ and Theorem \ref{theoremNn-1} is obtained for $n=2k+1,r=2, i=0.$

\section{Hyperpfaffians and Hyperdeterminants}
In this section we briefly discuss interactions between some invariants studied here in this paper with some  invariants  in the literature  namely  hyperdeterminants and hyperpfafians.
\subsection{Pfaffian of Skew-symmetric bilinear forms and the determinant}\label{pfaffian}
Let $h$ be a skew-symmetric bilinear form on the Euclidean vector space $(V,g)$. Then $h$ can be seen either as a usual  $2$-form  or as a $(1,1)$ double form.\\
Suppose $\dim V=n=2k$ is even.  We know already that $*\frac{h^n}{n!}$  is the determinant of $h$ once $h$ is seen as a $(1,1)$ double form, see section \ref{first-section}. However, if we perform the same operations on the $2$-form $h$ and of course with  the ordinary exterior product of forms and the usual Hodge star operator we obtain another invariant ${\rm Pf}(h)$ called the Pfaffian of $h$. Precisely, we have
\begin{equation}
{\rm Pf}(h)=*\frac{h^k}{k!}.
\end{equation}
It turns out that ${\rm Pf}(h)$ is a square root of the determinant of $h$ that is $({\rm Pf}(h))^2=\det h$. This identity can be quickly justified as follows: \\
We proceed by duality, in the case where $h$ is considered as a $2$-form, then $\frac{h^k}{k!}\otimes  \frac{h^k}{k!}$ is an $(n,n)$ double form. Since the space of $(n,n)$ double forms  on $V$ is $1$-dimensionnal vector space then it is proportionnal to the double $*\frac{h^n}{n!}$ where in the last expression $h$ is seen as a $(1,1)$ double form. It turns out that the previous two $(n,n)$ double forms are equal. To show the desired equation it suffices to compare the  image of the two double forms under the generalized Hodge star operator. From one hand we have
 \begin{equation*}
\ast\left( \frac{h^k}{k!}\otimes  \frac{h^k}{k!}\right)=\ast \frac{h^k}{k!}\otimes  \ast \frac{h^k}{k!}=({\rm Pf}(h))^2.
\end{equation*}
On the other hand we have
 \begin{equation*}
\ast\left( \frac{h^n}{n!}\right)=\det h.
\end{equation*}
 The  Pfaffian satisfies several similar properties to those of the determinant and one may use the exterior product of exterior forms (as we did here for double forms) to prove such properties.
\subsection{Pfaffian of $4$-forms}
Let  $\omega$ be a $4$-form on the Euclidean vector space $V$. Remark that $\omega$ can be naturally considered  as a symmetric $(2,2)$ double form.\\
Suppose $\dim V=n=4k$ is a multiple of $4$.  By definition we have $\frac{1}{(2k)!}*\omega^{2k}=h_n(\omega)$  is the  $h_n$ invariant of $\omega$ once  it is considered as a $(2,2)$ double form, see section \ref{second-section}. However, if we perform the same operations on the $4$-form $\omega$ and of course with  the ordinary exterior product of forms and the usual Hodge star operator we obtain a new invariant ${\rm Pf}(\omega)$, which we shall call the Pfaffian of $\omega$. Precisely, we set
\begin{equation}
{\rm Pf}(\omega)=*\frac{\omega^k}{k!}.
\end{equation}
Using the same argument as in the above section \ref{pfaffian} it is plausible that $({\rm Pf}(\omega))^2=h_n(\omega)$. 

\subsection{Pfaffians of  $2k$ forms}
We now define the Pfaffian for higher forms. Let  $\omega$ be a $2k$-form on an Euclidean vector space $(V,g)$ of finite dimension $n$. Remark that $\omega$ can be naturally considered  as a  $(k,k)$ double form.\\
Suppose $\dim V=n=2kq$ is a multiple of $2k$.  By definition we have $\ast \frac{\omega^{2q}}{(2q)!}=h_{(0,n)}(\omega)$  is the  $h_{(0,n)}$ invariant of $\omega$ once  it is considered as a $(k,k)$ double form, see section \ref{third-section}. However, if we perform the same operations on the $2k$-form $\omega$ and of course with  the ordinary exterior product of forms and the usual Hodge star operator we obtain another  invariant ${\rm Pf}(\omega)$, which we shall call the Pfaffian of $\omega$. Precisely, we set
\begin{equation}
{\rm Pf}(\omega)=\ast\frac{\omega^q}{q!}.
\end{equation}
Using the same simple argument as in section \ref{pfaffian} it is plausible that $\left({\rm Pf}(\omega)\right)^2=h_{(0,n)}(\omega)$. \\
It turns out that these Pfaffians coincide with the hyperpfaffians defined first in \cite{Barvi}, see also formula (79) in section 5.1 of  \cite{Luque}.
\subsection{Multiforms and Hyperdeterminants}
Hyperdeterminants of hypermatrices were introduced first by Cayley in 1843.
To introduce  hyperdeterminants we need first  to briefly  introduce a generalization of double forms namely multiforms.\\
Let $(V,g)$ be an Euclidean real vector space  of dimension n.  A $(k,...,k)$ multiform is by definition an element of the tensor product  $ \Lambda^{k}V^* \otimes ... \otimes  \Lambda^{k}V^*$. 
We define the exterior product of two multiforms in the natural way:    for two multiforms $\omega_1=\theta_1\otimes ... \otimes \theta_r$ and
    $\omega_2=\phi_1\otimes ...\otimes  \phi_r$, we set
    \begin{equation}
    \label{def:prod}
     \omega_1\omega_2= 
    (\theta_1\wedge \phi_1)\otimes ... \otimes (\theta_r\wedge \phi_r).
\end{equation}
We extend also the usual Hodge star operator to multiforms in the obvious way:
\[*\bigl(\phi_1\otimes ...\otimes  \phi_r\bigr)=(*\phi_1)\otimes ...\otimes  (*\phi_r).\]
Let  $\omega$ be  a  $(k,...,k)$ multiform  and suppose  that the dimension $n=pk$ is a multiple of $k$. We define the hyperdeterminant of $\omega$, denoted ${\rm{Det}}(\omega)$, to be the scalar
\begin{equation}
{\rm{Det}}(\omega)=*\frac{\omega^p}{p!}.
\end{equation}
Obviously, every tensor (hypermatrix)  $T$ can be seen as a $(1,1,...,1)$ multiform, and under this view the above definition of the hyperdeterminant ${\rm{Det}}(T)$ coincides with the standard one as in \cite{Barvi,Luque}.\\
Let now $\omega$ be a $k$ form such that $k=pr$ is a multiple of some positive integer $r$ and the dimension $n$ of $V$ is a multiple of $k$ say $n=qk$. Then $\omega$ can be naturally seen as a $(p,...,p)$ multiform ($r$-times). Then the  hyperdeterminant of the $(p,...,p)$ multiform $\omega$ is related to the hyperpfaffian of the $k$ form $\omega$ via the relation
\begin{equation}
\bigl({\rm Pf}(\omega)\bigr)^r={\rm{Det}}(\omega).
\end{equation}
We recover the results of the previous subsection for $r=2$.

\section{Final remarks and open question}
\subsection{Geometric applications of the general algebraic identities}
We have seen in the sequel of the paper that the identity $t_n(h)=0$ for a bilinear form leads to a linearized version of the Gauss-Bonnet theorem for compact hypersurfaces of the Euclidean space. Furthermore, for a $(2,2)$ double form $R$, the identity $T_n(R)=0$ , for $n$ even and in the context of Riemannian manifolds, is an infinitisimal form of  the general Gauss-Bonnet-Chern theorem for  compact manifolds. \\
It is then natural to ask what differential or topological consequences can be drawn from the identity $N_{n-1}(R)=0$ for a compact Riemannian manifold of odd dimension $n$?\\
Note that for a $3$-dimensionnal Riemannian manifold, the identity $N_{n-1}(R)=0$ for the Rieman curvature tensor $R$ is equivalent to the vanishing of the Weyl tensor.\\
The same question can be asked for the  other different  higher  algebraic identities established here in this paper.

\subsection{Spectrum of cofactor transformations of $(2,2)$ double forms.}
Let $R$  be a  symmetric  $(2,2)$ double form defined over  the $n$-dimensional Euclidean vector space  $(V,g)$. Denote by $\lambda_1,\lambda_2,...,\lambda_N$, where $N=\frac{n(n-1)}{2}$, the eigenvalues of the linear operator $\Lambda^{2}V\rightarrow \Lambda^{2}V$ that is canonically associated to $R$. For each $k$ with $1\leq k\leq n-2$, the exterior product $g^kR$ is a $(k+2,k+2)$ symmetric double form and therefore has real eigenvalues once considered as an operator  $\Lambda^{k+2}V\rightarrow \Lambda^{k+2}V$. The eigenvalues of the operator are expected to be  polynomials in the eigenvalues of $R$. The question here is to find  explicit formulas for the eigenvalues of $g^kR$ in terms of  $\lambda_1,\lambda_2,...,\lambda_N$.\\
In the case $k=n-2$ the answer is trivial as we have in this case one single eigenvalue and it is equal to $\sum_{i=1}^N\lambda_i$. \\
To motivate this question let us recall that the Weitzenb\"{o}ck transformation of order $p$, $2\leq p\leq n-2$, of the double form $R$ is given by \cite{Labbiweitz}
\[ {\mathcal N}_p(R)=\left( \frac{gcR}{p-1}-2R\right)\frac{g^{p-2}}{(p-2)!}.\]
The positivity of the transformation ${\mathcal N}_p(R)$ has important consequences in Riemannian geometry via the celebrated  Weitzenb\"{o}ck formula.\\
The true question is then to determine the eigenvalues of ${\mathcal N}_p(R)$ in terms of the eigenvalues of $R$.\\
More generally, what are the eigenvalues of the exterior products $g^pR^q$ and the contractions $c^pR^q$  in terms of  $\lambda_1,\lambda_2,...,\lambda_N$?\\
In particular, find explicit formulas for  the invariants $h_{2k}(R)$, for $k>1$,  in terms of  $\lambda_1,\lambda_2,...,\lambda_N$.

\subsection{Cofactor transformations vs. Gilkey's restriction map}
Following \cite{Gilkey-book,GPS} we define $\mathcal{ I}_{m,n}^{p+1}$ to be the space of invariant local formulas for symmetric
$(p,p)$ double forms that satisfy the first Bianchi identity and that are homogeneous of degree n in the derivatives
of the metric and which are defined in the category of $m$ dimensional Riemannian
manifolds. In particular $\mathcal{ I}_{m,n}^{1}=\mathcal{ I}_{m,n}$ is the the space of scalar invariant
local formulas and $\mathcal{ I}_{m,n}^{2}$ is the space  of symmetric
2-form valued invariants.\\
Recall that the homogeneity of order $n$ for $\omega=\omega(g)\in \mathcal{ I}_{m,n}^{p+1}$ is equivalent to
\[ \omega(\lambda^2g)=\frac{1}{\lambda^{n-2p}}\omega(g),\]
for all scalars $\lambda\not= 0.$\\
The last property implies in particular that if $\omega(g)\in \mathcal{ I}_{m,n}^{p+1}$ then its full contraction  $\ccc^p \omega(g)\in \mathcal{ I}_{m,n}$ and  $\ccc^{p-1} \omega(g)\in \mathcal{ I}_{m,n}^{2}$.\\
The Gilkey's restriction map  $r:\mathcal{ I}_{m,n}\rightarrow \mathcal{ I}_{m-1,n}$  is closely related to the adjuagate transformations  as we will explain below.\\
Let $\omega(g)\in \mathcal{ I}_{m,n}^{p+1}$, recall   that the  $(1,pq)$ cofactor  transformation of $\omega$ is  $h_{(1,pq)}(\omega)=\ast \frac{g^{n-pq-1}\omega^q}{(n-pq-1)!}\in \mathcal{ I}_{m,n}^{2}$. For a tangent vector $v$, we have
\[h_{(1,pq)}(\omega)(v,v)=\frac{g^{(n-1)-pq}\omega^q}{\left((n-1)-pq\right)!}(*v,*v).\]
That is the restriction of the invariant formula $Q=*\frac{g^{n-pq}\omega}{(n-pq)!}=\frac{c^{pq}\omega^q}{(pq)!}\in \mathcal{ I}_{m,n}$  to the $(n-1)$  dimensional orthogonal complement of  the vector $v$.\\
More generally, $ h_{(r,pq)}(\omega)$ is given by  $r$ successive applications of Gilkey's restriction map to the invariant formula $\frac{c^{pq}\omega}{(pq)!}\in \mathcal{ I}_{m,n}$.\\

\address{Mohammed Larbi Labbi\\
Mathematics Department\\
College of Science\\
University of Bahrain\\
32038 Bahrain.}\\
\email{ml.labbi@gmail.com}
\end{document}